\documentclass{amsart}
\usepackage{amsfonts,amsmath,amssymb,amscd,mathrsfs}
\usepackage{xcolor}
\usepackage{verbatim}
\usepackage{enumitem}
\usepackage{hyperref}

\newtheorem{theorem}{Theorem}[section]
\newtheorem{lemma}[theorem]{Lemma}
\newtheorem{corollary}[theorem]{Corollary}
\newtheorem{proposition}[theorem]{Proposition}
\theoremstyle{definition}

\newtheorem{example}[theorem]{Example}

\newtheorem{remark}[theorem]{Remark}

\def \D{\mathbb{D}}

\def \R{\mathbb{R}}
\def \C{\mathbb{C}}
\def \T{\mathbb{T}}

\newcommand{\clb}{\mathcal{B}}

\newcommand{\clm}{\mathcal{M}}
\newcommand{\cln}{\mathcal{N}}

\newcommand{\clk}{\mathcal{K}}

\newcommand{\dist}{\mathrm{dist}}

\newcommand{\essran}{\mathrm{ess\ ran}}
\newcommand{\essinf}{\mathrm{ess\ inf}}
\newcommand{\esssup}{\mathrm{ess\ sup}}

\newcommand{\vp}{\varphi}

\newcommand{\ol}{\overline}
\newcommand\restr[2]{\ensuremath{\left.#1\right|_{#2}}}

\numberwithin{equation}{section}



\begin{document}
\title[On the minimum modulus of dual truncated Toeplitz operators]{On the minimum modulus of dual truncated Toeplitz operators}
\author{Sudip Ranjan Bhuia}
\address{Shiv Nadar IoE, NH91, Tehsil Dadri, Greater Noida, Gautam Buddha Nagar, Uttar Pradesh-201314, India}
\email{sudipranjanb@gmail.com; sudip.bhuia@snu.edu.in}

\author{Ramesh Golla}
\address{Department of Mathematics, IIT Hyderabad, 
Kandi, Sangareddy, Telangana, India, 502\,284}
\email{rameshg@math.iith.ac.in}

\author{Puspendu Nag}
\address{Department of Mathematics, IIT Hyderabad, 
Kandi, Sangareddy, Telangana, India, 502\,284}
\email{ma23resch11003@iith.ac.in; puspomath@gmail.com}

\date{\today}

\begin{abstract}
This article provides a systematic investigation of the minimum modulus of dual truncated Toeplitz operators (DTTOs) $D_{\varphi}$ acting on the orthogonal complement of the model space $\mathcal{K}_u^{\perp}$, where $u$ is a nonconstant inner function and $\varphi \in L^\infty(\T)$. We first establish an explicit formula for the minimum modulus of the compressed shift $S_u$ and its dual $D_u$ in terms of $|u(0)|$, and prove that the minimum is always attained. For normal DTTOs, we derive sharp spectral bounds utilizing the essential range of the symbol and characterize the conditions under which $m(D_{\varphi})$ coincides with the essential infimum of $|\varphi|$. In the general setting, for unimodular $\vp$, we obtain exact formulas and two sided estimates for $m(D_{\varphi})$ by analyzing the norms of associated Toeplitz and Hankel operators restricted to the model space. Finally, we provide several concrete examples to illustrate our results.
\end{abstract}

\maketitle	

\section{Introduction}The study of operators on function spaces has long been a cornerstone of modern analysis, with Toeplitz and Hankel operators serving as primary examples of operators with rich structural and spectral properties \cite{Douglas1998, Peller:Book}. A significant development in this field was the introduction of truncated Toeplitz operators (TTOs) by Sarason \cite{Sarason2007}. These operators arise as compressions of multiplication operators to the model space $\clk_u=H^2 \ominus uH^2$, associated with a nonconstant inner function $u$. The structural richness of TTOs has generated extensive research interest owing to their deep connections with complex symmetry, interpolation problems, and model theory \cite{Garcia2016}. More recently, Ding and Sang \cite{DingSang2018} introduced the class of \textit{dual truncated Toeplitz operators} (DTTOs). For $\vp\in L^\infty$, the DTTO $D_\vp$ is defined on $\clk_u^\perp = uH^2 \oplus H^2_{-}$ by
\[D_\vp = (I - P_{\clk_u}) \restr{M_\vp}{\clk_u^\perp},\]
where $M_\vp$ is the multiplication operator on $L^2(\T)$. The structural theory of DTTOs has seen rapid advancement; for instance, Sang, Qin, and Ding \cite{Sang:Qin:Ding} established a Brown-Halmos type theorem and showed that DTTOs can be realized as $2 \times 2$ Toeplitz--Hankel operator matrices. Further characterizations of their commutants, invariant subspaces, and essential commutation properties were investigated in \cite{Li:Sang:Ding, Wang:Zhao:Zheng}, while Gu \cite{Gu2021} provided fundamental characterizations of their normality and spectral inclusions. Although substantial progress on the algebraic and spectral characterizations of DTTOs has been extensively investigated, the quantitative behavior of these operators specifically their extremal properties remains an active area of inquiry. Two central concepts in this regard are \textit{norm attaining} and \textit{minimum attaining} operators. An operator $T \in \clb(H, K)$ is said to be norm attaining if there exists a unit vector $x$ such that $\|Tx\| = \|T\|$, a property with a long history beginning with the work of Lindenstrauss \cite{Lindenstrauss1963}. Analogously, the \textit{minimum modulus} of $T$ is defined as
\[m(T) = \inf\{\|Tx\| : \|x\|=1\},\]
which measures the distance of the operator from the set of non-invertible operators \cite{GindlerTaylor}. We say $T$ is minimum attaining if this infimum is achieved at some vector in the unit sphere. For instance, Carvajal and Neves investigated operators that achieve their norm or minimum in \cite{Carvajal2012, CarvajalNeves}. Unlike the operator norm, which is strictly positive for every nonzero operator, the minimum modulus may vanish for a nonzero operator. This distinction reveals that norm attainment and minimum attainment exhibit fundamentally different structural behavior. In particular, every compact operator on a Hilbert space attains its norm, whereas a compact operator attains its minimum modulus if and only if it fails to be injective (see \cite[Proposition 1.3]{Carvajal2012}). The study of minimum attaining operators has gained significant attention, particularly regarding their density in $\clb(H,K)$ \cite{SHK:RG2018} and their structural characterization in various operator classes \cite{CarvajalNeves, GaneshRameshSukumar2015}. In the context of function spaces, the extremal behavior of Toeplitz and Hankel operators has been extensively explored. Toeplitz operators that attain their norm were first studied by Brown and Douglas in \cite{Brown:Douglas}.  More recently, Ramesh and Sequeira characterized absolutely norm attaining and absolutely minimum attaining Toeplitz and Hankel operators in \cite{RameshShanola2022, RameshShanola2023}. For DTTOs, the investigation of such properties is relatively new. In \cite{SudipPuspendu2026}, the authors classified the extremal sets and characterized norm attaining DTTOs in terms of the factorization of their inducing symbols, building upon the foundational work in \cite{DingSang2018, Gu2021}.The primary objective of this paper is to provide a systematic investigation into the minimum modulus of dual truncated Toeplitz operators. We first establish a formula for the minimum modulus of the compressed shift $S_u$ on the model space, which serves as a foundation for analyzing the dual compressed shift $D_u$ \cite{CamaraRoss2021}. For the latter, we derive an explicit expression for its minimum modulus in terms of the scalar $|u(0)|$. These results characterize the behavior of compressed shifts acting both on the model space $\clk_u$ and its orthogonal complement $\clk_u^\perp$. A key component of this analysis is Theorem \ref{thm:kernel-dim-equality}, where we prove that if $\dim N(T) = \dim N(T^*)$, then $m(T) = m(T^*)$. This result provides a useful generalization of spectral properties found in \cite{BalaRamesh2020}. As a significant application, we show that this equality of minimum moduli holds for all complex symmetric operators, a class that naturally includes dual truncated Toeplitz operators \cite{Gu2021}. We then extend our study to normal DTTOs, establishing sharp bounds in terms of the essential range of the symbol. Finally, for general symbols, we derive exact formulas for $m(D_\vp)$ using associated Toeplitz and Hankel operator norms \cite{Peller:Book}, supported by concrete examples that demonstrate the sharpness of our estimates.

The remainder of the paper is structured as follows. In Section 2, we establish the necessary preliminaries concerning the geometry of model spaces and the fundamental properties of the minimum and reduced minimum moduli. Section 3 is dedicated to the study of the compressed shift and its dual counterpart, providing explicit computations for their minimum moduli. Section 4 investigates the class of normal dual truncated Toeplitz operators, and provides spectral bounds based on the essential range of the inducing symbol. Finally, in Section 5, we treat the general case. For unimodular symbols, we derive exact formulas and two sided estimates for the minimum modulus of DTTOs  (cf. Theorem \ref{thm:min-modulus-Dphi}) through the analysis of associated Toeplitz and Hankel operators. Furthermore, several concrete examples are discussed to illustrate the sharpness of the obtained bounds and to demonstrate the application of the derived formulas in specific settings. We conclude by discussing the minimum modulus of $B_\vp$ for $\vp\in H^\infty$.

\section{Preliminaries}

Let $ H$ be an infinite dimensional separable complex Hilbert space, and let $\clb( H)$ denote the algebra of all bounded linear operators on $H$. The \emph{kernel} (or null space) and \emph{range} (or range space) of $T$ are denoted by $N(T)$ and $R(T)$, respectively. The basic relations between range and kernel include $R(T)^\perp = N(T^*),\, \ol{R(T)} = N(T^*)^\perp.$ The operator $T$ is said to be \emph{bounded below} if there exists $c>0$ such that
\[
\|Tx\|\geq  c\|x\| \quad\text{for all } x\in H.
\]
This is equivalent to $N(T)=\{0\}$ and $R(T)$ being closed.

A \emph{conjugation} on a Hilbert space $H$ is an antilinear isometry $C: H \to H$ that is also an involution, that is, $C^2 = I$. A bounded linear operator $T \in \clb(H)$ is \emph{complex symmetric} if there exists a conjugation $C$ such that $T = CT^*C$ (cf. \cite{GarciaPutinar2006}).

The \emph{spectrum} of $T\in\clb( H)$ is the set
\[
\sigma(T)=\{\lambda\in\C : T-\lambda I \text{ is not invertible in } \clb( H)\}.
\]

The spectrum is a nonempty compact subset of $\C$ contained in the closed disc $\{z\in\C: |z|\leq  \|T\|\}$.

The \emph{point spectrum} (or set of eigenvalues) of $T$ is defined by
\[
\sigma_p(T)
:=
\{\lambda\in\C:\ N(T-\lambda I)\neq\{0\}\},
\]
that is, $\lambda\in\sigma_p(T)$ if and only if there exists a nonzero vector
$x\in H$ such that $Tx=\lambda x$.
The \emph{approximate point spectrum} of $T$ is the set
\[
\sigma_{\mathrm{ap}}(T)
:=
\left\{\lambda\in\C:\ \exists\,\{x_n\}\subset H,\ \|x_n\|=1,\ 
\|(T-\lambda I)x_n\|\to0\right\}.
\]
For any two fixed non-zero vectors $x, y \in  H$, the rank-one operator, often denoted as $x \otimes y$, is defined by 
\[(x \otimes y)(f) = \langle f, y \rangle x, \quad \text{for all } f \in  H.\]
For $T\in\clb(H)$, the \emph{minimum modulus} of $T$ is defined by
\[
m(T)=\inf_{\|x\|=1}\|Tx\|.
\]
The following basic properties will be used throughout (cf. \cite{CarvajalNeves, GindlerTaylor}):

\begin{proposition}\label{prop:basic-formulas}
Let $T\in\clb( H)$. Then the following statements are true:
\begin{enumerate}
\item $m(T)=0$ if and only if $T$ is not bounded below; equivalently,
      if and only if $R(T)$ is not closed or $N(T)\neq\{0\}$.

\item $m(T^*T)=[m(|T|)]^2$.

\item If $V: H\to  H$ is unitary, then $m(VTV^*)=m(T)$.

\item If $N(T)\neq\{0\}$, then $m(T)=0$, and the minimum modulus is automatically attained.

\item $m(T)=\inf \sigma(|T|)$ equivalently, $m(T)^2=\inf\sigma(T^*T)$.
\item $\lambda\in \sigma_{\mathrm{ap}}(T)$ if and only if $m(T-\lambda I)=0$.
\end{enumerate}
For normal $T\in\clb(H)$, by \cite[Proposition 2.1]{Ramesh2014}, we have $m(T)=\dist (0,\sigma(T))$.
\end{proposition}

Recall that an operator $T\in\clb(H)$ is said to be \emph{norm attaining} if there exists
$x_0\in S_{ H}$ such that
\[
\|Tx_0\|=\|T\|,
\]
where $S_{ H}$ denotes the unit sphere of $ H$.
We denote by $\cln(H)$ the class of all norm attaining operators on $H$.

An operator $T\in \clb(H)$ is said to be \emph{minimum attaining} if there exists
$x_0\in S_{H}$ such that
\[
\|Tx_0\|
=
m(T)
=
\inf\{\|Tx\|:\ x\in S_{ H}\}.
\]
We denote by $\clm(H)$ the set of all minimum attaining operators on $H$.

When an operator has nontrivial kernel, its minimum modulus is necessarily zero. In such situations it is natural to study the behavior of $T$ on the orthogonal complement of its kernel.

For $T\in\clb( H)$, the \emph{reduced minimum modulus} of $T$ is defined by
\[
\gamma(T)
=
\inf\{\|Tx\| : x\in N(T)^\perp,\, \|x\|=1\}.
\]

Equivalently,
\[
\gamma(T)=\inf\sigma\,\left(\restr{|T|}{N(T)^\perp}\right),
\qquad |T|=(T^*T)^{\frac{1}{2}}.
\]

Thus $\gamma(T)>0$ if and only if $T$ is bounded below on $N(T)^\perp$, that is, there exists $c>0$ such that $\|Tx\|\geq  c\|x\|$ for all $x\in N(T)^\perp$. It is evident that $m(T)\leq \gamma(T)$ and in particular, when $N(T)=\{0\}$, we have $\gamma(T)=m(T)$.

Let $(X,\Sigma,\mu)$ be a measure space. For a measurable function 
$f:X\to\R\cup\{\pm\infty\}$, the \emph{essential infimum} of $f$ is defined by
\begin{align*}
\essinf f
:&=\sup\{a\in\R:\ f(x)\geq  a \text{ for almost every\ }x\in X\}\\
&=\sup\{a\in\R:\ \mu(\{x:\ f(x)<a\})=0\}.
\end{align*}
If $\vp\in L^\infty(X,\mu)$, its \emph{essential range} is the set
\[
\essran(\vp)
:=\{\lambda\in\C:\ \forall\varepsilon>0,\
\mu(\{x:\ |\vp(x)-\lambda|<\varepsilon\})>0\}.
\]
Note that
\[
0\in \essran(\vp)
\quad\Longleftrightarrow\quad
\essinf|\vp|=0,
\]
while $0\notin\essran(\vp)$ if and only if $\essinf|\vp|>0$, equivalently, $\vp^{-1}\in L^\infty(X,\mu)$. It is well known that for $\vp\in L^\infty$, we have $\sigma (\vp)=\essran (\vp)$ and $\|\vp\|_\infty=\sup \{|\lambda|:\,\lambda\in\essran (\vp)  \}$ (cf. \cite[Chapter 2]{Douglas1998}).

Let $\D$ denote the open unit disc and $\T=\partial\D$ the unit circle. We write $L^2:=L^2(\T)$ for the usual Hilbert space of square integrable measurable functions on $\T$ with respect to the normalized Lebesgue measure $\mu$. The Hardy space $H^2:=H^2(\T)$ is the closed subspace of $L^2$ consisting of functions whose negative Fourier coefficients vanish. 

We denote by $L^\infty :=L^\infty(\T)$, the Banach space of all essentially bounded $\mu$-measurable functions on $\T$. Let $H^\infty$ denote the Banach algebra of all bounded analytic functions on the unit disc $\D$, equipped with the norm
\[
\|f\|_\infty=\sup_{z\in\D}|f(z)|.
\]
In terms of boundary functions, often we write $H^\infty:=H^2 \cap L^\infty$. A function $u\in H^\infty$ is called an \emph{inner function} if
\[
|u(\zeta)|=1 \quad \text{for almost every } \zeta\in\T .
\]
By the Maximum Modulus Theorem, every inner function $u$ satisfies $|u(z)|\leq  1$ for all $z\in\D$. Moreover, if $u$ is nonconstant, then $|u(z)| < 1$ for all $z \in \D$.

Let $P:L^2\to L^2$ be the orthogonal projection onto $H^2$, and set $P_{-}:=I-P.$ For $\vp\in L^\infty$, let $M_\vp$ denote multiplication by $\vp$ on $L^2$, and the \emph{Toeplitz operator} with symbol $\vp$ is the bounded operator
\[
T_\vp : H^2 \to H^2, \qquad
T_\vp f = P(\vp f), \qquad f\in H^2 .
\]

The \emph{Hankel operator} with symbol $\vp\in L^\infty$ is the operator
\[
H_\vp : H^2 \to H^2_{-}, \qquad
H_\vp f = P_{-}(\vp f), \qquad f\in H^2 .
\]
Throughout, $S$ denotes the unilateral shift and $S^*$ its adjoint on $H^2$.

Let $u$ be a nonconstant inner function on the unit disc $\D$.  
The associated \emph{model space} is defined by
\[
\clk_u = H^2\ominus uH^2.
\]
It is a closed subspace of $H^2$ that is invariant under the backward shift $S^*$. We denote the orthogonal projection of $L^2$ onto $\clk_u$ by $P_{\clk_u}$.

The orthogonal complement of $\clk_u$ in $L^2$ admits the decomposition
\[
\clk_u^\perp = uH^2 \oplus H^2_- ,\text{ where } H^2_- = P_-(L^2)=\overline{zH^2} .
\]

The space $\clk_u$ is a reproducing kernel Hilbert space.
For each $\lambda\in\D$, the reproducing kernel of $\clk_u$ at $\lambda$ is
\[
k_\lambda^u(z)=\frac{1-\ol{u(\lambda)}\,u(z)}{1-\ol{\lambda}z},
\qquad z\in\D .
\]
That is, for every $f\in \clk_u$,
\[
f(\lambda)=\langle f, k_\lambda^u\rangle .
\]

In particular, at $\lambda=0$, 
\begin{equation}\label{eq:kernel-lambda-zero}
k_0^u(z)=1-\ol{u(0)}\,u(z),
\qquad
\|k_0^u\|^2 = k_0^u(0)=1-|u(0)|^2 .
\end{equation}

The reproducing kernels play a central role in the analysis of operators on $\clk_u$. For example, the compressed shift $S_u=P_{\clk_u}\restr{S}{\clk_u}$ satisfies the rank--one defect identities
\begin{equation}\label{eq:defect-identities}
    I_{\clk_u}-S_u^*S_u=S^*u\otimes S^*u,\qquad
I_{\clk_u}-S_uS_u^*=k_0^u\otimes k_0^u,
\end{equation}
which is fundamental in studying norm and minimum modulus properties of $S_u$.

The \emph{dual truncated Toeplitz operator} (DTTO) with symbol $\vp$ and inner function $u$ is the operator
\[
D_\vp = (I-P_{\clk_u})M_\vp\big|_{\clk_u^\perp},
\]
acting on $\clk_u^\perp$. The operator $D_\vp$ is bounded if and only if $\vp\in L^\infty$ and in this case $\|D_\vp\|=\|\vp\|_\infty$ (cf. \cite[Property 2.1]{DingSang2018}). Also, we have $D_\vp^*=D_{\bar{\vp}}$. Often we write $P_{\clk_u^\perp}:=I-P_{\clk_u}$, the orthogonal projection of $L^2$ onto $\clk_u^\perp$.

The operator $D_\vp$ is complex symmetric with respect to the conjugation $C_u:\clk_u^\perp\rightarrow\clk_u^\perp$ defined by 
\begin{equation}\label{conjugation}
C_uf=u\ol{zf},\quad f\in \clk_u^\perp,
\end{equation} that is, $C_u D_\vp C_u=D^*_\vp$ (cf. \cite[Theorem 2.7]{Gu2021}).

Following \cite{Sang:Qin:Ding}, there exists a unitary map
\begin{equation}\label{eq:unitary-U}
U=
\begin{pmatrix}
M_u & 0\\ 
0 & I
\end{pmatrix}
:
H^2\oplus H^2_{-} \longrightarrow uH^2\oplus H^2_{-}=\clk_u^\perp.
\end{equation}

under which $D_\vp$ is unitarily equivalent to the block operator
\begin{equation}\label{eq:DTTO-block}
U^*D_\vp U=
\begin{pmatrix}
T_\vp & H_{u\bar\vp}^*\\ 
H_{u\vp} & S_\vp
\end{pmatrix},
\end{equation}
where
 $S_\vp=P_{-}\restr{M_\vp}{H^2_{-}}$ is the \emph{dual Toeplitz operator}. The anti-unitary operator $V$ on $L^2$ defined by $(Vf)(w)=\bar{w}\ol{f(w)},\, w\in\T,\, f\in L^2$ satisfies the following properties:
\[
V=V^{-1},
\qquad
VT_\vp=S_{\ol{\vp}} V.
\]

In particular, when $\vp(z)=z$, the dual compressed shift is denoted by $D_u$. By \cite[Lemma~4.1]{CamaraRoss2021}, $D_u$ has the following operator matrix representation
\begin{equation}\label{matrix:DCS}
U^*D_uU=
\begin{pmatrix}
S & H_{u\bar{z}}^*\\ 
0 & Q
\end{pmatrix}:=\widetilde{D_u},
\qquad
Qg=P_{-}(zg),
\end{equation}
which plays a central role in the present work. The spectrum of the \emph{dual compressed shift
operator} $D_u$ satisfies the following dichotomy (cf. \cite[Proposition 3.4]{CamaraRoss2021}):
\[
\sigma(D_u)=
\begin{cases}
\ol{\D}, & \text{if } u(0)=0,\\[4pt]
\T, & \text{if } u(0)\neq 0.
\end{cases}
\]

Moreover, by \cite[Proposition 8]{Camara2020}, the defect identities for $D_u$  are given by
\begin{equation}\label{eq:defect-Du}
D_uD_u^*
=
I-(1-|u(0)|^2)\,u\otimes u ,
\end{equation}
\begin{equation}\label{eq:defect-Du-star}
D_u^* D_u = I - (1-|u(0)|^2) \bar{z} \otimes \bar{z}.
\end{equation}

\begin{corollary}
    Let $u$ be an inner function. The dual truncated Toeplitz operator $D_u$ on $\clk_u^\perp$ is an isometry if and only if $u$ is a constant function of modulus $1$.
\end{corollary}

\section{Compressed shifts and dual compressed shifts}
In this section, we study the minimum modulus of compressed shift type operators. We begin with the compressed shift $S_u$ acting on the model space $\clk_u$. In this setting, the rank one defect identities satisfied by $S_u$ allow for an explicit analysis of the spectrum and, consequently, of the minimum modulus. These results provide a useful prototype for the arguments that follow.
We then turn to the dual compressed shift acting on $\clk_u^\perp$, whose
structure is more involved but can be effectively analyzed using the block operator representation and the spectral information obtained in the compressed shift case. This approach clarifies both the parallels and the distinctions between the compressed and dual compressed settings.

\subsection{Minimum modulus of compressed shifts}
We begin this section with a simple but essential lemma concerning the spectral behavior of rank one perturbations of the identity. Despite its elementary nature, the lemma plays a crucial role in our analysis and will be invoked repeatedly to compute the spectrum, and therefore the minimum modulus, of operators arising from rank-one defect identities such as those satisfied by compressed shifts.

\begin{lemma}\label{lem:spectrum-perturbation}
Let $ H$ be a complex Hilbert space and $x\in H$ be a nonzero vector. For $\alpha,\beta\in\C\setminus\{0\}$, define the rank one perturbation of the identity operator
\[
T := \alpha I + \beta\, x\otimes x.
\]

Then the spectrum of $T$ is given by
\[
\sigma(T)=\{\alpha,\, \alpha+\beta\|x\|^{2}\}.
\]
\end{lemma}

\begin{proof}
Easily follows by the spectral mapping theorem.
\end{proof}

\begin{theorem}\label{thm:kernel-dim-equality}
Let $T \in \clb( H)$ such that $\dim N(T) = \dim N(T^*)$. Then
\[
m(T) = m(T^*).
\]
\end{theorem}

\begin{proof}
We shall proceed by considering two cases: first, when the kernel is non-trivial, and second, when the kernel is trivial.

\textbf{Case 1: $\dim N(T) > 0$.} 
If the kernel is non-trivial, there exists a unit vector $x_0 \in N(T)$ such that $Tx_0 = 0$. By the definition of the minimum modulus,
\[
m(T) = \inf_{\|x\|=1} \|Tx\| \le \|Tx_0\| = 0,
\]
which implies $m(T)=0$. Since $\dim N(T^*) = \dim N(T) > 0$, the adjoint operator $T^*$ also has a non-trivial kernel. Thus, there exists a unit vector $y_0 \in N(T^*)$ such that $T^*y_0 = 0$, implying $m(T^*) = 0$. In this case, $m(T) = m(T^*) = 0$.

\textbf{Case 2: $\dim N(T) = 0$.} 
Then $N(T) = \{0\}$, and the condition $\dim N(T)=\dim N(T^*)$ implies that $N(T^*) = \{0\}$. We distinguish two sub-cases based on whether the range of $T$ is closed:
\begin{enumerate}
    \item If $R(T)$ is not closed, then $T$ is not bounded below, which implies $m(T) = 0$. Since $N(T^*) = \{0\}$, the range $R(T)$ is dense. If $R(T)$ is dense but not closed, then $R(T^*)$ is also not closed, which implies $m(T^*) = 0$.
    \item If $R(T)$ is closed, then since $N(T^*) = \{0\}$, we have $R(T) = N(T^*)^\perp =  H$. Thus, $T$ is invertible, and by the Inverse Mapping Theorem, $T^{-1} \in \clb(H)$. 
\end{enumerate}

In the invertible case, the minimum modulus is given by
\[
m(T) = \frac{1}{\|T^{-1}\|}.
\]
Since $T$ is invertible, its adjoint $T^*$ is also invertible with $(T^*)^{-1} = (T^{-1})^*$. We have
\[
m(T^*) = \frac{1}{\|(T^*)^{-1}\|} = \frac{1}{\|(T^{-1})^*\|} = \frac{1}{\|T^{-1}\|} = m(T).
\]
Thus, $m(T) = m(T^*)$.
\end{proof}

\begin{remark}
   The above theorem partially generalizes \cite[Corollary 3.14]{BalaRamesh2020}.
\end{remark}

The following result follows from an application of Theorem \ref{thm:kernel-dim-equality} and the established properties of complex symmetric operators in \cite[Proposition 1]{GarciaPutinar2006}.
\begin{proposition}\label{prop:adj-bddbelow-invrt-positiveminmod}
    Let $T\in\clb(H)$ be a complex complex symmetric operator. Then the following are true:
    \begin{enumerate}
        \item $m(T)=m(T^*)$;
        \item $T$ is bounded below if and only if $T$ is invertible;
        \item $m(T) > 0$ if and only if $T$ is invertible.
    \end{enumerate}
\end{proposition}

It is worth emphasizing the contrast between the unilateral shift $S$ on $H^2$ and its compression $S_u$ to the model space $\clk_u$. While $S$ is an isometry and hence satisfies $m(S)=1$, its adjoint $S^*$ has nontrivial kernel, so $m(S^*)=0$. Thus the minimum modulus is not preserved under taking adjoints in general. In sharp contrast, the compressed shift $S_u$ exhibits a fundamentally different behavior: both $S_u$ and $S_u^*$ are contractions with rank one defect operators, and their minimum moduli are governed by the scalar quantity $|u(0)|$. The following theorem shows that, unlike the case of the unilateral shift, the operators $S_u$ and $S_u^*$ have the same strictly positive minimum modulus whenever $u(0)\neq0$, and moreover both operators attain their minimum modulus.

\begin{theorem}\label{thm:minmod-compressed-shift}
Let $u$ be a nonconstant inner function. Then $m(S_u)=|u(0)|=m(S_u^*).$
Moreover, both $S_u$ and $S_u^*$ attain their minimum modulus.
\end{theorem}
\begin{proof}
By the defect identities \eqref{eq:defect-identities}, we have both $S_u^*S_u$ and $S_uS_u^*$ are rank--one perturbations of the identity on $\clk_u$.
Applying Lemma~\ref{lem:spectrum-perturbation} with $\alpha=1,\,\beta=-1$ gives
\[
\sigma(S_u^*S_u)=\{\,1,\ 1-\|S^*u\|^2\,\},
\qquad
\sigma(S_uS_u^*)=\{\,1,\ 1-\|k_0^u\|^2\,\}.
\]
Note that $S^*u=C_uk_0^u$, where $C_u$ is given by \eqref{conjugation}.  Thus using \eqref{eq:kernel-lambda-zero} and the fact that
\[
\|S^*u\|^2=\|C_u k_0^u\|^2=\|k_0^u\|^2,
\]
we obtain
\[
\sigma(S_u^*S_u)=\{1,\,|u(0)|^2\},
\qquad
\sigma(S_uS_u^*)=\{1,\,|u(0)|^2\}.
\]
Now $u$ being nonconstant, we have $|u(0)|<1$. Consequently,
\[
\inf\sigma(S_u^*S_u)=|u(0)|^2,
\qquad
\inf\sigma(S_uS_u^*)=|u(0)|^2.
\]
Since $m(S_u)^2=\inf\sigma(S_u^*S_u)$ and $m(S_u^*)^2=\inf\sigma(S_uS_u^*)$, we conclude that
\[
m(S_u)=|u(0)|=m(S_u^*).
\]
Finally, by \eqref{eq:defect-identities},
\[
S_u^*S_u(S^*u)=(1-\|S^*u\|^2)\,S^*u=|u(0)|^2\,S^*u,
\]
so $|u(0)|^2$ is an eigenvalue of $S_u^*S_u$. Hence by \cite[Proposition 3.1]{GaneshRameshSukumar2015} and \cite[Proposition 3.2]{GaneshRameshSukumar2018}, $S_u$ attains its minimum modulus.
Similarly,
\[
S_uS_u^*(k_0^u)=(1-\|k_0^u\|^2)\,k_0^u=|u(0)|^2\,k_0^u,
\]
so $|u(0)|^2$ is an eigenvalue of $S_uS_u^*$ and therefore, $S_u^*$ attains its minimum modulus.
\end{proof}

\subsection{Minimum modulus of dual compressed shifts}

Let $u$ be a nonconstant inner function. Recall that the dual truncated Toeplitz
operator (DTTO) with symbol $\vp\in L^\infty$ is
\[
D_\vp=(I-P_{\clk_u})\restr{M_\vp}{\clk_u^\perp},
\]
acting on $\clk_u^\perp=uH^2\oplus H^2_{-}$.
In this subsection we consider the distinguished symbol $\vp(z)=z$.

\begin{lemma}\label{lem:Q-ker}
Let $Q:H^2_{-}\to H^2_{-}$ be defined by
\[
Qg=P_{-}(zg),\qquad g\in H^2_{-}.
\]
Then $Q$ is a partial isometry with
$N(Q) = \mathrm{span}\{\bar z\}.$
Consequently, the \emph{reduced} minimum modulus of $Q$ equals $1$.
\end{lemma}

\begin{proof}
Since $H^2_-=\ol{zH^2}$, every $g\in H^2_-$ admits a Fourier expansion
\[
g=\sum_{n\geq  1} a_n \bar{z}^{n}
\quad\text{with}\qquad
\|g\|^2=\sum_{n\geq  1}|a_n|^2.
\]
Multiplying by $z$ gives
\[
zg=\sum_{n\geq  1} a_n \bar{z}^{(n-1)} = a_1 + \sum_{n\geq  2} a_n \bar{z}^{(n-1)}.
\]
Since $P_-$ is the orthogonal projection onto $H^2_-$, we get
\[
Qg=P_{-}(zg)=\sum_{n\geq  2} a_n \bar{z}^{(n-1)}.
\]
It follows that $Qg=0$ if and only if $a_n=0$ for all $n\geq  2$, that is,
$g=a_1 \bar{z}$. Hence $N(Q)=\mathrm{span}\{\bar{z}\}$.

Now suppose $g\perp \bar{z}$. Then $a_1=0$, and so
\[
\|Qg\|^2=\sum_{n\geq  2}|a_n|^2=\sum_{n\geq  1}|a_n|^2=\|g\|^2.
\]
Thus $Q$ is an isometry on $N(Q)^\perp$. In particular, $Q$ is a partial
isometry (an operator that is an isometry on $N(Q)^\perp$).

Since $\|Qg\|=\|g\|$ for all $g\in N(Q)^\perp$, we get $\gamma(Q)=1$.
\end{proof}

\begin{proposition}\label{prop:kernel-Au}
Let $\widetilde{D_u}
=
\begin{pmatrix}
S & H_{u\bar z}^{*} \\[4pt]
0 & Q
\end{pmatrix}
\quad \text{acting on } H^{2}\oplus H^{2}_{-},$
where $S$ is the unilateral shift on $H^{2}$ and $Qg=P_{-}(zg)$ on $H^{2}_{-}$.
Then
\[
N(\widetilde{D_u})=
\begin{cases}
\{0\}\oplus \operatorname{span}\{\bar z\}, & \text{if } u(0)=0,\\[4pt]
\{0\}, & \text{if } u(0)\neq 0.
\end{cases}
\]
\end{proposition}

\begin{proof}
Let $f\oplus g\in H^{2}\oplus H^{2}_{-}$. Then $f\oplus g\in N(\widetilde{D_u})$ if and only if
\[
Sf+H_{u\bar z}^{*}g=0
\quad\text{and}\quad
Qg=0.
\]
From Lemma \ref{lem:Q-ker}
\[
N(Q)=\operatorname{span}\{\bar z\}.
\]
Now we consider two cases which exhaust all possibilities.

\textbf{Case 1: $u(0)=0$.}
In this case $H_{u\bar z}^{*}=0$, so the kernel equations reduce to
\[
Sf=0
\quad\text{and}\quad
Qg=0.
\]
Since $S$ is injective on $H^{2}$, we obtain $f=0$, while $g\in N(Q)
=\operatorname{span}\{\bar z\}$. Therefore
\[
N(\widetilde{D_u})=\{0\}\oplus \operatorname{span}\{\bar z\}.
\]

\textbf{Case 2: $u(0)\neq 0$.}
Then by \cite[Theorem 4.2]{CamaraRoss2021}, we have
\[
H_{u\bar z}^{*}
=
\ol{u(0)}(1\otimes \bar z).
\]
From $Qg=0$ we have $g=c\bar z$ for some $c\in\C$. Then
\[
H_{u\bar z}^{*}g
=
\ol{u(0)}\,\langle c\bar z,\bar z\rangle\,1
=
\ol{u(0)}\,c.
\]
The first kernel equation becomes
\[
Sf+\ol{u(0)}\,c=0.
\]
Since $Sf\in zH^{2}$ has zero constant term, this equality forces $c=0$.
Hence $g=0$, and then $Sf=0$ implies $f=0$. Thus $N(\widetilde{D_u})=\{0\}$.

The conclusion follows.
\end{proof}

\begin{theorem}\label{thm:minmod-S-H-Q}
Let $u$ be a nonconstant inner function. 
Then
\[
m( D_u)=
\begin{cases}
0, & u(0)=0,\\ 
|u(0)|, & u(0)\neq 0,
\end{cases}
\]
and in both cases $ D_u$ attains its minimum modulus.
\end{theorem}

\begin{proof}
By Lemma~\ref{lem:Q-ker} we have $N(Q)=\C\bar z$. Moreover, $Q$ is an isometry
on $N(Q)^\perp$. By \cite[Theorem~4.2]{CamaraRoss2021}, the adjoint Hankel operator $H_{u\bar z}^*$
satisfies
\[
H_{u\bar z}^*\equiv 0 \text{ if }u(0)=0,
\qquad
H_{u\bar z}^*=\ol{u(0)}(1\otimes \bar z) \text{ if }u(0)\neq 0.
\]
In the latter case $H_{u\bar z}^*$ is rank one and
\[
H_{u\bar z}^*g=\ol{u(0)}\,\langle g,\bar z\rangle\,1,
\qquad g\in H^2_{-}.
\]

\textbf{Case 1: $u(0)=0$.}
Then $H_{u\bar z}^*\equiv 0$, so $ \widetilde{D_u}=S\oplus Q$. By Proposition~\ref{prop:kernel-Au}, $N(\widetilde{D_u})\neq \{0\}$, and therefore
$m(\widetilde{D_u})=0$. In fact, for $x=0\oplus\bar z$, we have $\|x\|=1$ and
\[
\widetilde{D_u} x=0\oplus Q\bar z=0.
\]

\smallskip 
\textbf{Case 2: $u(0)\neq 0$.}
Set 
\[\alpha:=u(0) \text{ and } B:=H_{u\bar z}^*=\ol{\alpha}(1\otimes \bar z).\]
Let $e:=\bar z\in H^2_{-}$. Then $\|e\|=1$ and $e\in N(Q)$. Moreover,
$Be=\ol{\alpha}$.

For $f\in H^2$, we have
\[
\widetilde{D_u}(f\oplus e)=(Sf+Be)\oplus Qe=(Sf+\ol{\alpha})\oplus 0.
\]
Since $R(S)=zH^2$ and $1\perp zH^2$, we get
\[
\|Sf+\ol{\alpha}\|^2=\|Sf\|^2+|\alpha|^2=\|f\|^2+|\alpha|^2.
\]
Therefore, 
\[
\frac{\|\widetilde{D_u}(f\oplus e)\|^2}{\|(f\oplus e)\|^2}
=
\frac{\|f\|^2+|\alpha|^2}{\|f\|^2+1}.
\]
Since $0<|\alpha|=|u(0)|< 1$, we have
\[
\frac{\|f\|^2+|\alpha|^2}{\|f\|^2+1}-|\alpha|^2
=
\frac{\|f\|^2+|\alpha|^2-|\alpha|^2(\|f\|^2+1)}{\|f\|^2+1}
=
\frac{(1-|\alpha|^2)\|f\|^2}{\|f\|^2+1}\geq 0.
\]
Hence $\frac{\|f\|^2+|\alpha|^2}{\|f\|^2+1}\geq |\alpha|^2$, and consequently, for $f\oplus e\neq 0$, we have
\[
\frac{\|\widetilde{D_u} (f\oplus e)\|}{\|f\oplus e\|}\geq |\alpha|,\qquad\text{for all}\,f\in H^2.
\]
In particular, taking $f=0$ gives $\|0\oplus e\|=1$ and $\|\widetilde{D_u}(0\oplus e)\|=|\alpha|$. Thus
\begin{equation}\label{eq:upperbound}
m(\widetilde{D_u})\leq  \|\widetilde{D_u}(0\oplus e)\|=|\alpha|.
\end{equation}

Now let $f\oplus g\in H^2\oplus H^2_{-}$ with $f\oplus g\neq0$.
Decompose $H^2_-=N(Q)\oplus N(Q)^\perp$ and write $g$ as
\[
g=ce+h,\text{ where } c\in\C,\,h\in N(Q)^\perp.
\]
Then $Qe=0$ and $\|Qh\|=\|h\|$, while $Bh=\ol{\alpha}\langle h,e\rangle\,1=0$ and $B(ce)=\ol{\alpha}c$. Thus we get
\[
\widetilde{D_u}(f\oplus g)=(Sf+\ol{\alpha}c)\oplus Qh.
\]
Therefore, using $Sf\perp 1$,
\[
\|\widetilde{D_u}(f\oplus g)\|^2
=\|f\|^2+|\alpha|^2|c|^2+\|h\|^2,
\qquad
\|(f\oplus g)\|^2=\|f\|^2+|c|^2+\|h\|^2.
\]
Letting $A:=\|f\|^2+\|h\|^2 > 0$, we obtain
\[
\frac{\|\widetilde{D_u}(f\oplus g)\|^2}{\|f\oplus g\|^2}
=
\frac{A+|\alpha|^2|c|^2}{A+|c|^2}
=
|\alpha|^2+\frac{A(1-|\alpha|^2)}{A+|c|^2}
\geq  |\alpha|^2,
\]
since $|\alpha|=|u(0)|<  1$. Hence 
\begin{equation}\label{eq:lowerbound}
    m(\widetilde{D_u})\geq |\alpha|.
\end{equation}
Combining (\ref{eq:upperbound}) and (\ref{eq:lowerbound}) gives $m(\widetilde{D_u})=|\alpha|=|u(0)|$.

Finally, the unit vector $0\oplus e$ satisfies
\[
\|\widetilde{D_u}(0\oplus e)\|=\|\ol{\alpha}\oplus 0\|=|\alpha|=m(\widetilde{D_u}).
\]
Thus in each case $\widetilde{D_u}$ attains its minimum modulus. Since minimum attainment is preserved by unitary equivalence, by Proposition~\ref{prop:basic-formulas} and \eqref{matrix:DCS}, $D_u$ is minimum attaining. 
\end{proof}
\begin{remark}
The above proof relies entirely on the definition of the minimum modulus. An alternative and shorter proof can be obtained by combining the defect identity \eqref{eq:defect-Du-star} with the Lemma \ref{lem:spectrum-perturbation}.

\end{remark}

\section{Minimum modulus of normal DTTOs}

For a nonconstant inner function $u$, and $\vp\in L^\infty$, the dual truncated Toeplitz operator $D_\vp$ on $\clk_u^\perp$ satisfies $\|D_\vp\|=\|\vp\|_\infty$. The minimum modulus is given by
\[
m(D_\vp)
:=
\inf\left\{\|D_\vp x\|:\ x\in \clk_u^\perp,\ \|x\|=1\right\}.
\]
Recall that $m(D_\vp)>0$ if and only if $D_\vp$ is bounded below.
Since $D_\vp$ is complex symmetric with respect to the conjugation $C_u$, the operator $D_\vp$ is bounded below if and only if $D_\vp$ is invertible (cf. Proposition \ref{prop:adj-bddbelow-invrt-positiveminmod}).

\begin{proposition}\label{prop:min-modulus-inequality}
Let $u$ be a nonconstant inner function, and let $D_\vp$ be a normal dual truncated Toeplitz operator on $\clk_u^\perp$ with symbol $\vp \in L^\infty$. Then the minimum modulus $m(D_\vp)$ satisfies the following inequalities:
\begin{equation}
\dist\bigl(0, \operatorname{convex}(\essran(\vp))\bigr) \le m(D_\vp) \le \essinf_{z \in \T} |\vp(z)|.
\end{equation}
\end{proposition}

\begin{proof}
Since $D_\vp$ is a normal, its minimum modulus is equal to the distance from the origin to its spectrum:
\begin{equation}\label{eq:dist-spec}
m(D_\vp) = \dist(0, \sigma(D_\vp)).
\end{equation}

We utilize the spectral inclusions for dual truncated Toeplitz operators. First, by \cite[Corollary 4.3]{Gu2021}, we have
\[
\essran(\vp) \subseteq \sigma(D_\vp).
\]
Second, by \cite[Theorem 4.5]{Gu2021}, the spectrum is contained in the convex hull of the essential range
\[
\sigma(D_\vp) \subset \operatorname{convex}(\essran(\vp)).
\]

Since $\essran(\vp) \subseteq \sigma(D_\vp)$, the distance from the origin to the spectrum cannot exceed the distance from the origin to the subset $\essran(\vp)$. Thus
\[
m(D_\vp) \le \dist(0, \essran(\vp)) = \essinf_{z \in \T} |\vp(z)|.
\]

Since $\sigma(D_\vp) \subset \operatorname{convex}(\essran(\vp))$, the distance from the origin to the spectrum must be at least the distance from the origin to the containing set $\operatorname{convex}(\essran(\vp))$. Thus
\[
m(D_\vp) \geq \dist\bigl(0, \operatorname{convex}(\essran(\vp))\bigr).
\]

Combining these inequalities completes the proof.
\end{proof}

\begin{remark}
   For $\vp\in L^\infty$, if $D_\vp$ is self-adjoint, then $\vp$ is a real-valued function. Consequently, $\operatorname{convex}(\essran(\vp))=[\essinf\ \vp, \esssup\ \vp]$, and so $\dist(0,\operatorname{convex}(\essran(\vp)))=\essinf_{z\in \T} |\vp(z)|$. In this case, by Proposition \ref{prop:min-modulus-inequality}, we obtain 
   \[m(D_\vp)=\essinf_{z\in \T} |\vp(z)|.\]
\end{remark}
\begin{example}
We apply Proposition \ref{prop:min-modulus-inequality} to the symbol
\[
\vp(z) = \psi(z) + \beta,
\]
where $\beta \in \C$ and $\psi$ is defined as
\[
\psi(e^{it}) =
\begin{cases}
1 & \text{if } 0 \le t < \pi, \\
-1 & \text{if } \pi \le t < 2\pi.
\end{cases}
\]

The essential range of $\vp$ consists of exactly two points
\[
\essran(\vp) = \{ \beta - 1, \beta + 1 \}.
\]
Consequently, the essential infimum is the minimum of the moduli of these two points
\[
\essinf_{z \in \T} |\vp(z)| = \min \{ |\beta - 1|, |\beta + 1| \}.
\]

The convex hull of these two points is the line segment connecting them, that is,
\[
\operatorname{convex}(\essran(\vp)) = [\beta - 1, \beta + 1].
\]

According to Proposition \ref{prop:min-modulus-inequality}, the minimum modulus is bounded as follows
\[
\dist\bigl(0, [\beta - 1, \beta + 1]\bigr) \le m(D_\vp) \le \min \{ |\beta - 1|, |\beta + 1| \}.
\]

\textbf{Specific Case:} Let $\beta = 3i$.
The upper bound is
\[
\essinf |\vp| = \min \{ |3i - 1|, |3i + 1| \} = \sqrt{10}.
\]
For the lower bound, the segment connects $-1+3i$ and $1+3i$. The point closest to the origin is the midpoint $3i$. Thus
\[
\dist\bigl(0, \operatorname{convex}(\essran(\vp))\bigr) = |3i| = 3.
\]
The inequality yields $3 \le m(D_\vp) \le \sqrt{10}$. 
\end{example}

\begin{proposition}\label{cor:min-modulus-convex}
Let $D_\vp$ be a normal dual truncated Toeplitz operator with symbol $\vp\in L^\infty$. If the essential range $\essran(\vp)$ is a convex set (for example, if $\vp$ is continuous), then the minimum modulus is exactly the essential infimum of the symbol
\[
m(D_\vp)=\essinf_{z\in\T}|\vp(z)|.
\]
In this case, $D_\vp$ is bounded below if and only if $\vp$ is invertible in $L^\infty$.
\end{proposition}

\begin{proof}
For any normal DTTO, the minimum modulus is determined by the distance between origin and the spectrum $\sigma(D_\vp)$ and since $\essran (\vp)$ is convex, $\sigma (D_\vp)=\essran (\vp)$ (cf. \cite[Corollary 4.6]{Gu2021}).
 Therefore,
\[
m(D_\vp) =\dist\left( 0, \sigma (D_\vp)\right)= \dist\bigl(0, \essran(\vp)\bigr) = \inf\{|\lambda| : \lambda \in \essran(\vp)\}.
\]
It is well known that the infimum of the modulus of values in the essential range is equal to the essential infimum of the function
\[
\inf\{|\lambda| : \lambda \in \essran(\vp)\} = \essinf_{z \in \T} |\vp(z)|.
\]
Thus, $m(D_\vp)=\essinf_{z\in \T}|\vp(z)|$.
\end{proof}

\begin{example}
Consider the symbol $\vp(z) = z + \bar{z} + 2i$.
Note that $z + \bar{z} = 2\Re(z)$ is real-valued continuous function on $\T$. Thus, $\vp$ is of the form
\[
\vp = \psi + \beta, \quad \text{where } \psi(z)=2\Re(z) \text{ is real, and } \beta = 2i.
\]
By the characterization of normal dual truncated Toeplitz operators, $D_\vp$ is normal.

First note that the essential range of $\vp$ on the unit circle is the image of the interval $[-1, 1]$ under the map $t \mapsto 2t + 2i$:
    \[
    \essran(\vp) = \{ 2x + 2i : x \in [-1, 1] \}.
    \]

    The minimum modulus is given by
    \[
    m(D_\vp) = \dist(0, \sigma(D_\vp)) = |2i| = 2.
    \]

   By applying formula described in Proposition \ref{cor:min-modulus-convex}, we also obtain
    \[
    m(D_\vp) =\displaystyle \essinf_{z\in\T} |z + \bar{z} + 2i| = \essinf_{x\in[-1,1]} |2x + 2i|.
    \]
    The modulus is $\sqrt{(2x)^2 + 2^2} = \sqrt{4x^2 + 4}$.
    This function is minimized when $x=0$, yielding
    \[
    \sqrt{0 + 4} = 2.
    \]
We conclude $m(D_\vp)=2$, verifying the Proposition \ref{cor:min-modulus-convex}.
\end{example}

\section{DTTOs with unimodular symbols }
If $\vp\in L^\infty$ with $|\vp|=1$ a.e, then the multiplication operator $M_\vp$ is an isometry. With the decomposition $L^2 = \clk_u \oplus \clk_u^\perp$, the multiplication operator $M_\vp$ has the block matrix form
\[M_\vp = \begin{pmatrix} A_\vp & B_{\bar{\vp}}^* \\ B_\vp & D_\vp \end{pmatrix}.\]
Here, $A_\vp = P_{\clk_u}\restr{ M_\vp }{\clk_u}$ and $B_\vp = P_{\clk_u^\perp} \restr{M_\vp }{\clk_u}$.

Since $M_\vp^* M_\vp = I$, we have
\[\begin{pmatrix} A_\vp^* & B_\vp^* \\ B_{\bar{\vp}} & D_\vp^* \end{pmatrix} \begin{pmatrix} A_\vp & B_{\bar{\vp}}^* \\ B_\vp & D_\vp \end{pmatrix} = \begin{pmatrix} I_{\clk_u} & 0 \\ 0 & I_{\clk_u^\perp} \end{pmatrix}.\]

We obtain
\[A_\vp^* A_\vp + B_\vp^* B_\vp = I_{\clk_u},\]

or
\begin{equation}\label{eq:Bphi-Aphi}
    B_\vp^* B_\vp = I_{\clk_u} - A_\vp^* A_\vp.
\end{equation}

Taking the norm, we get
\[\|B_\vp\|^2 = \|B_\vp^* B_\vp\| = \|I_{\clk_u} - A_\vp^* A_\vp\|.\]
But we have
\[\|I_{\clk_u} - A_\vp^* A_\vp\| = 1 - \inf_{\|x\|=1} \langle A_\vp^* A_\vp x, x \rangle = 1 - m(A_\vp)^2.\]

Therefore,
\begin{equation}\label{minmod TTO}
    \|B_\vp\| = \sqrt{1 - m(A_\vp)^2}.
\end{equation}

For any $f\in L^2$, we have the following identity
\begin{equation}
\|f\|^2=\|P_{\clk_u} f\|^2 + \|(I-P_{\clk_u})f\|^2.
\end{equation}

The following proposition establishes a precise link between the minimum modulus of a dual truncated Toeplitz operator $D_\vp$ and the norm of the operator $B_{\bar{\vp}}$ under the assumption that the inducing symbol $\vp$ is unimodular.

\begin{proposition}\label{prop:minmod-Dphi-norm-Bphi}
Let $\vp\in L^\infty$ with $|\vp|=1$ a.e. on $\T$. Then 
\begin{equation}
    m(D_\vp)=\sqrt{1-\|B_{\bar{\vp}}\|^2}.
\end{equation}
\end{proposition}

\begin{proof}
    Let $h\in \clk_u^\perp$ with $\|h\|=1$. Then
    \begin{align*}
        \|D_\vp h\|^2&=\|(I-P_{\clk_u})\vp h\|^2\\
        &=\|\vp h\|^2 - \|P_{\clk_u} \vp h\|^2 \\
        &=1-\|B^*_{\bar{\vp}}h\|^2.
    \end{align*}

    Thus $\inf_{\|h\|=1} \|D_\vp h\|^2=1-\sup_{\|h\|=1}\|B^*_{\bar{\vp}}h\|^2$, that is, $m(D_\vp)^2=1-\|B^*_{\bar{\vp}}\|^2$. Consequently, 
\begin{equation}
        m(D_\vp)=\sqrt{1-\|B_{\bar{\vp}}\|^2}.\qedhere
\end{equation}
\end{proof}

\begin{remark}
    By Proposition \ref{prop:minmod-Dphi-norm-Bphi} and the relation \eqref{minmod TTO}, for unimodular $\vp\in L^\infty$, we get $m(A_\vp)=m(D_{\bar{\vp}})=m(D_\vp)$. 
\end{remark}

The following proposition establishes a fundamental structural link between truncated Toeplitz operators (TTOs) and their dual counterparts. In the study of operators on model spaces, the interaction between the space $\clk_u$ and its orthogonal complement $\clk_u^\perp$ is often complex; however, this result demonstrates that the dual truncated Toeplitz operator $D_\vp$, when pulled back to the model space via the isometry $M_u$ and its adjoint $M_{\bar{u}}$, can be expressed entirely in terms of operators acting on $\clk_u$. Specifically, for analytic symbols, this relation collapses into a direct equality, effectively allowing us to transfer spectral and norm properties from the dual space back to the more manageable TTOs. This correspondence is a vital tool for characterizing the minimum modulus and invertibility of block operators where these two classes of operators appear simultaneously.

\begin{proposition}\label{prop:relation-A-D}
Let $u$ be a nonconstant inner function, and let $\vp\in L^\infty$. Then the following identity holds on $\clk_u$:
\begin{equation}\label{eq:main-relation}
\restr{P_{\clk_u}\,M_{\bar u}\,D_\vp\,M_u}{\clk_u}=  A_{\vp}-A_{\bar u}A_{\vp u}.
\end{equation}
In particular, if $\vp\in H^\infty$, then $A_{\vp u}=0$ and
\begin{equation}\label{eq:analytic-reduction}
\restr{P_{\clk_u}\,M_{\bar u}\,D_\vp\,M_u}{\clk_u}
=
A_\vp .
\end{equation}
\end{proposition}

\begin{proof}
Let $f\in\clk_u$. Since $uf\in uH^2\subset\clk_u^\perp$, we have
\[
D_\vp(uf)
=
P_{\clk_u^\perp}(\vp\,u f)
=
(I-P_{\clk_u})(\vp\,u f).
\]
Applying $M_{\bar u}$ and then projecting onto $\clk_u$ yields
\begin{align*}
P_{\clk_u}M_{\bar u}D_\vp M_u f
&=
P_{\clk_u}M_{\bar u}\left(\vp\,u f\right)-P_{\clk_u}M_{\bar u}P_{\clk_u}\left(\vp\,u f\right)\\
&=
P_{\clk_u}\left((\bar u\,\vp\,u)f\right)-P_{\clk_u}M_{\bar u}\left(A_{\vp u}f\right)\\
&=
P_{\clk_u}\left(\vp f\right)-P_{\clk_u}M_{\bar u}\left(A_{\vp u}f\right).
\end{align*}
The first term equals $A_{\vp}f$ by definition. Since
$A_{\vp u}f\in\clk_u$, the second term can be written as
\[
P_{\clk_u}M_{\bar u}(A_{\vp u}f)
=
A_{\bar u}(A_{\vp u}f)
=
(A_{\bar u}A_{\vp u})f.
\]
This proves \eqref{eq:main-relation}.

If $\vp\in H^\infty$, then $\vp u\in uH^\infty$, so
$(\vp u)f\in uH^2$ for every $f\in\clk_u$, which implies
$A_{\vp u}=0$. Hence \eqref{eq:main-relation} reduces to
\eqref{eq:analytic-reduction}.
\end{proof}

\begin{lemma}\label{lem:norm-Bphi}
Let $u$ be a nonconstant inner function. Consider the operator $B_\vp: \clk_u \to \clk_u^\perp$ defined by $B_\vp f = P_{\clk_u^\perp}(\vp f),\, f\in \clk_u$.

\begin{enumerate}
    \item For any symbol $\vp \in L^\infty$, $B_\vp$ admits the orthogonal decomposition
    \begin{equation}\label{eq:Bphi-decomp}
    B_\vp = M_u \restr{T_{\bar{u}\vp}}{\clk_u} + \restr{H_\vp}{\clk_u},
    \end{equation}
    where $T_{\bar{u}\vp}$ is the Toeplitz operator with symbol $\bar{u}\vp$ and $H_\vp$ is the Hankel operator with symbol $\vp$. Consequently, the norm is given by
    \begin{equation}\label{eq:norm-Bphi-Linf}
    \|B_\vp\| = \sup_{f \in \clk_u, \|f\|=1} \sqrt{ \|T_{\bar{u}\vp} f\|^2 + \|H_\vp f\|^2 }.
    \end{equation}

    \item  If $\vp \in H^\infty$, then $H_\vp = 0$ and the range of $B_\vp$ is contained in $uH^2$. The expression simplifies to
    \begin{equation}\label{eq:Bphi-Toeplitz}
    B_\vp = M_u \restr{T_{\bar u\vp}}{\clk_u}.
    \end{equation}
    In particular, $\|B_\vp\| = \left\|\restr{T_{\bar u\vp}}{\clk_u}\right\|$, and $B_\vp=0$ if and only if $\vp \clk_u \subset \clk_u$ (equivalently, $\vp$ is constant).
\end{enumerate}
\end{lemma}

\begin{proof}
To prove (1), let $f \in \clk_u$. Since $\clk_u^\perp = uH^2 \oplus H^2_-$, we write $P_{\clk_u^\perp} = P_{uH^2} + P_{-}$, so
\[
B_\vp f = P_{uH^2}(\vp f) + P_{-}(\vp f).
\]
For the first term, using the identity $P_{uH^2} = M_u P M_{\bar{u}}$, we have
\[
P_{uH^2}(\vp f) = M_u P(\bar{u}\vp f) = M_u T_{\bar{u}\vp} f.
\]
For the second term, by definition, $P_{-}(\vp f) = H_\vp f$.
Substituting these back yields decomposition \eqref{eq:Bphi-decomp}.
Since $uH^2$ and $H^2_-$ are orthogonal, and $M_u$ is isometric, we have
\[
\|B_\vp f\|^2 = \|M_u T_{\bar{u}\vp} f\|^2 + \|H_\vp f\|^2 = \|T_{\bar{u}\vp} f\|^2 + \|H_\vp f\|^2.
\]
Taking the supremum over unit vectors yields \eqref{eq:norm-Bphi-Linf}.

To prove (2), assume that $\vp \in H^\infty$. Then $\vp f \in H^2$, so $P_-(\vp f) = 0$, which implies $H_\vp f = 0$ and $B_\vp f \in uH^2$. The decomposition \eqref{eq:Bphi-decomp} thus reduces to \eqref{eq:Bphi-Toeplitz}. The norm equality follows immediately. Finally, $B_\vp = 0$ is equivalent to $P_{uH^2}(\vp f) = 0$ for all $f \in \clk_u$, which means $\vp f \in H^2 \cap (uH^2)^\perp = \clk_u$, that is, $\vp \clk_u \subset \clk_u$ which is equivalent to $\vp$ being constant (cf. \cite[Theorem 6.12]{Garcia2016}).
\end{proof}

\begin{theorem}\label{thm:min-modulus-Dphi}
Let $u$ be a nonconstant inner function, and let $\vp \in L^\infty$ be a unimodular function. The minimum modulus of $D_\vp$ is given by
\begin{equation}\label{eq:mDphi}
m(D_\vp) = \sqrt{ 1 - \|B_{\bar{\vp}}\|^2 } = \sqrt{ 1 - \sup_{f \in \clk_u, \|f\|=1} \left( \|T_{\ol{u\vp}} f\|^2 + \|H_{\bar{\vp}} f\|^2 \right) }.
\end{equation}
Furthermore, $m(D_\vp)$ satisfies 
\begin{equation}\label{eq:mDphi-bounds}
\begin{split}
&\sqrt{ \max\left(0, 1 - \left( \|\restr{T_{\ol{u\vp}}}{\clk_u}\|^2 + \|\restr{H_{\bar{\vp}}}{\clk_u}\|^2 \right) \right) }\\
&\leq m(D_\vp) \leq
\min \left( \sqrt{1 - \|\restr{T_{\ol{u\vp}}}{\clk_u}\|^2}, \; \sqrt{1 - \|\restr{H_{\bar{\vp}}}{\clk_u}\|^2} \right).
\end{split}
\end{equation}
\end{theorem}

\begin{proof}
By Proposition \ref{prop:minmod-Dphi-norm-Bphi}, we have
\[
m(D_\vp) = \sqrt{1 - \|B_{\bar{\vp}}\|^2}.
\]
From Lemma \ref{lem:norm-Bphi}, applying the formula to the symbol $\bar{\vp}$, we have 
\[
\|B_{\bar{\vp}}\|^2 = \sup_{f \in \clk_u, \|f\|=1} \left( \|T_{\bar{u}\bar{\vp}} f\|^2 + \|H_{\bar{\vp}} f\|^2 \right).
\]
Since $\bar{u}\bar{\vp} = \ol{u\vp}$, we substitute this directly to obtain the exact formula \eqref{eq:mDphi}.

To derive the lower bound, we use the property that the supremum of a sum is bounded by the sum of the supremum
\[
\sup_{f \in \clk_u, \|f\|=1} \left( \|T_{\ol{u\vp}} f\|^2 + \|H_{\bar{\vp}} f\|^2 \right) 
\le \sup_{f\in \clk_u,\|f\|=1} \|T_{\ol{u\vp}} f\|^2 + \sup_{f\in \clk_u,\|f\|=1} \|H_{\bar{\vp}} f\|^2.
\]
This implies
\[
\|B_{\bar{\vp}}\|^2 \le \|T_{\ol{u\vp}}|_{\clk_u}\|^2 + \|H_{\bar{\vp}}|_{\clk_u}\|^2.
\]
Substituting this inequality into the expression $m(D_\vp) = \sqrt{1 - \|B_{\bar{\vp}}\|^2}$, establishing the lower bound. (We include the maximum with 0 to account for cases where the sum of squared norms exceeds $1$).

To derive the upper bound, we observe that for any non-negative functions $F(f)$ and $G(f)$, we have
\[
\sup_{f} (F(f) + G(f)) \geq \sup_{f} F(f) \quad \text{and} \quad \sup_{f} (F(f) + G(f)) \geq \sup_{f} G(f).
\]
Therefore, $\|B_{\bar{\vp}}\|^2 \geq \|T_{\ol{u\vp}}|_{\clk_u}\|^2$ and $\|B_{\bar{\vp}}\|^2 \geq \|H_{\bar{\vp}}|_{\clk_u}\|^2$.
 Thus
\[
m(D_\vp) \le \sqrt{1 - \|\restr{T_{\ol{u\vp}}}{\clk_u}\|^2} \quad \text{and} \quad m(D_\vp) \le \sqrt{1 - \|\restr{H_{\bar{\vp}}}{\clk_u}\|^2}.
\]
Taking the minimum of these two inequalities gives the upper bound in \eqref{eq:mDphi-bounds}.
\end{proof}

\begin{example}
Consider the inner function $u(z) = z^2$ and the symbol $\vp(z) = z$. The corresponding model space is $\clk_u = \operatorname{span}\{1, z\}$. Note that $\vp \in L^\infty$ and $|\vp|=1$ a.e. on $\T$. If $f \in \clk_u$ be a unit vector, then $f(z) = c_0 + c_1 z$ with $|c_0|^2 + |c_1|^2 = 1,\text{ where } c_0,c_1\in \C$.
We now compute $H_{\bar{\vp}} f = P_{-}(\bar{z}(c_0 + c_1 z)) = P_{-}(c_0 \bar{z} + c_1) = c_0 \bar{z}$. Therefore, $\|H_{\bar{\vp}} f\|^2 = |c_0|^2$. Now $\ol{u\vp} = \bar{z}^3$.
Then $T_{\ol{u\vp}} f = P(\bar{z}^3(c_0 + c_1 z)) = P(c_0 \bar{z}^3 + c_1 \bar{z}^2) = 0$. Thus, $\|T_{\ol{u\vp}} f\|^2 = 0$.

Substituting these into the formula, we get
\[
m(D_u) = \sqrt{1 - \sup_{|c_0|^2+|c_1|^2=1} (|c_0|^2 + 0)} = \sqrt{1 - 1} = 0.
\]

\end{example}

\begin{example}
Consider the inner function $u(z) = z^2$ (so $\clk_u = \operatorname{span}\{1, z\}$) and the unimodular symbol $\vp(z)$ defined by 
\[
\vp(z) = \bar{z} \left( \frac{z - \alpha}{1 - \alpha z} \right), \quad \text{where } \alpha \in (0,1).
\]
 We compute $m(D_\vp)$.

Now
\[
\bar{\vp}(z) = z \left( \frac{1 - \alpha z}{z - \alpha} \right).
\]
For $|z|=1 > \alpha$, we have the expansion $\frac{1}{z-\alpha} = \frac{\bar{z}}{1 - \alpha \bar{z}} = \sum_{k=0}^\infty \alpha^k \bar{z}^{k+1}$.
Thus
\[
\bar{\vp}(z) = z(1 - \alpha z) \sum_{k=0}^\infty \alpha^k \bar{z}^{k+1} = (z - \alpha z^2) (\bar{z} + \alpha \bar{z}^2 + \alpha^2 \bar{z}^3 + \cdots).
\]
This implies
\[
\bar{\vp}(z) = \underbrace{-\alpha z + (1-\alpha^2)}_{\text{Analytic part } (\in H^2)} + \underbrace{(1-\alpha^2) \sum_{k=1}^\infty \alpha^k \bar{z}^k}_{\text{Co-analytic part } (\in H^2_-)}.
\]

Now, let $f \in \clk_u$ be a unit vector $f(z) = c_0 + c_1 z$ with $|c_0|^2 + |c_1|^2 = 1$.

\[
H_{\bar{\vp}}(c_0 + c_1 z) = P_{-} \left[ \left( (1-\alpha^2)\sum_{k=1}^\infty \alpha^k \bar{z}^k \right) (c_0 + c_1 z) \right].
\]
Note that
\begin{align*}
H_{\bar{\vp}}(1) &= (1-\alpha^2) (\alpha \bar{z} + \alpha^2 \bar{z}^2 + \cdots) = S_0, \\
H_{\bar{\vp}}(z) &= P_{-}(1-\alpha^2) (\alpha + \alpha^2 \bar{z} + \cdots) =(1-\alpha^2)(\alpha^2 \bar{z} + \alpha^3 \bar{z}^2 + \cdots) = \alpha S_0.
\end{align*}
Thus, $H_{\bar{\vp}} f = (c_0 + \alpha c_1) S_0$.
Consequently
\begin{equation}
\|H_{\bar{\vp}} f\|^2 = |c_0 + \alpha c_1|^2 \|S_0\|^2,
\end{equation}
where

\[
\|S_0\|^2 = (1-\alpha^2)^2 \sum_{k=1}^\infty (\alpha^2)^k = (1-\alpha^2)^2 \frac{\alpha^2}{1-\alpha^2} = \alpha^2(1-\alpha^2).
\]

Next
\[
\ol{u\vp} = \bar{z}^2 [ -\alpha z + (1-\alpha^2) + \cdots ] = -\alpha \bar{z} + (1-\alpha^2)\bar{z}^2 + \cdots\in H^2_{-}.
\]
 Thus $T_{\ol{u\vp}}(1) = P(\ol{u\vp}) = 0$, and
\[
T_{\ol{u\vp}}(z) = P(\ol{u\vp} z) = P(-\alpha + (1 -\alpha^2)\bar{z} + \cdots) = -\alpha.
\]
So $T_{\ol{u\vp}} f = - \alpha c_1$. Therefore, 
\begin{equation}
\|T_{\ol{u\vp}} f\|^2=\alpha^2 |c_1|^2.
\end{equation}
Thus, we have the following minimum modulus
\[
m(D_\vp) = \sqrt{1 - \sup_{|c_0|^2+|c_1|^2=1} (\alpha^2|c_1|^2 +|c_0 + \alpha c_1|^2\alpha^2(1-\alpha^2) )} .
\]

\end{example}

\begin{example}
Let $u$ be a nonconstant inner function and consider the DTTO $D_{\vp}$ defined on $\clk_u^\perp$ with the symbol $\vp = u$. Since $u$ is an inner function, $|\vp|=1$ a.e. on $\T$, satisfying the unimodularity condition. We apply the Theorem \ref{thm:min-modulus-Dphi} to compute $m(D_\vp)$.

\[m(D_\vp) = \sqrt{1 - \|B_{\bar{u}}\|^2} = \sqrt{ 1 - \sup_{f \in \clk_u, \|f\|=1} \left( \|T_{\ol{u^2}} f\|^2 + \|H_{\bar{u}} f\|^2 \right) }.\]

Let $f \in \clk_u= H^2\cap (uH^2)^\perp$. Since $u^2H^2 \subset uH^2$, it follows that $f \perp u^2 H^2$. Since $M_u$ is an isometry, we have $\langle \bar{u}^2 f, h \rangle = \langle f, u^2 h \rangle = 0$ for all $h \in H^2$. This implies $\bar{u}^2 f \in H^2_-$, and thus $T_{\ol{u^2}} f = P(\bar{u}^2 f) = 0$ for all $f \in \clk_u$. Next, for any $f \in \clk_u$ with $\|f\|=1$, we have
\[\|H_{\bar{u}} f\|^2 = \|P_-(\bar{u}f)\|^2 = \|\bar{u}f\|^2 - \|P(\bar{u}f)\|^2 = 1 - \|P(\bar{u}f)\|^2=1.\]

Therefore,

\[m(D_\vp) = \sqrt{ 1 - \sup_{f \in \clk_u, \|f\|=1} \left(  1 - \|P(\bar{u}f)\|^2 \right) }=0.\]

\end{example}

\subsection{Some special symbol classes}

In this subsection, we provide explicit formulas for the minimum modulus of dual truncated Toeplitz operators under specific structural conditions on the symbol $\vp$. We begin by recalling standard properties of Toeplitz operators and their kernels.

\begin{proposition}\cite[Proposition 7.5]{Douglas1998} \label{prop:toeplitz-identities}
Let $\vp \in L^\infty$ and let $\psi, \bar{\theta} \in H^\infty$. Then the following algebraic identities hold:
\begin{enumerate}
    \item $T_\theta T_\vp = T_{\theta \vp}$;
    \item $T_\vp T_\psi = T_{\vp \psi}$.
\end{enumerate}
\end{proposition}

\begin{proposition}\cite[Proposition 5.8]{Garcia2016} \label{prop:kernel-toeplitz}
If $\vp \in H^\infty \setminus \{0\}$ and $\eta$ is an inner factor of $\vp$, then the kernel of the adjoint operator satisfies $N(T^*_\vp) = \clk_\eta$.
\end{proposition}

Our primary focus is the exact calculation of $m(D_\vp)$ using the formula derived in Theorem \ref{thm:min-modulus-Dphi}.
For a unimodular symbol $\vp$, and for all $f\in \clk_u$ with $\|f\|=1$, and using $1=\|\bar{\vp}f\|^2 = \|T_{\bar{\vp}}f\|^2 + \|H_{\bar{\vp}}f\|^2$, we have
\begin{equation}
\|T_{\ol{u\vp}} f\|^2 + \|H_{\bar{\vp}} f\|^2=\|T_{\ol{u\vp}} f\|^2 + 1 - \|T_{\bar{\vp}} f\|^2.
\end{equation}

The following corollary characterizes the cases where the general bounds collapse to exact equalities.

\begin{corollary} \label{cor:exact-equalities}
Let $u$ be a nonconstant inner function, and let $\vp \in L^\infty$ be unimodular. The following equalities hold:
\begin{enumerate}
    \item If $\vp$ is co-analytic ($\vp = \bar{\psi}$ for $\psi \in H^\infty$), then $m(D_{\bar{\psi}}) = \sqrt{1 - \|\restr{T_{\bar{u}\psi}}{\clk_u}\|^2}$.
    
    \item If $\restr{T_{\ol{u\vp}}}{\clk_u} = 0$ (e.g., if $\ol{u\vp} \in H^2_{-}$), then $m(D_\vp) = \sqrt{1 - \|\restr{H_{\bar{\vp}}}{\clk_u}\|^2}$.
    
    \item If there exists a unit vector $f_0 \in \clk_u$ that is a common maximizing vector for both $\|T_{\ol{u\vp}} f\|$ and $\|H_{\bar{\vp}} f\|$, then
    \[ m(D_\vp) = \sqrt{ 1 - \left( \|\restr{T_{\ol{u\vp}}}{\clk_u}\|^2 + \|\restr{H_{\bar{\vp}}}{\clk_u}\|^2 \right) }. \]
\end{enumerate}
\end{corollary}

\begin{proof}
(1): If $\vp = \bar{\psi}$ with $\psi \in H^\infty$, then the symbol for the Hankel operator in \eqref{eq:mDphi} is $\bar{\vp} = \psi$. Since $\psi$ is analytic, $H_\psi = 0$, yielding the result.

(2): If $T_{\ol{u\vp}}|_{\clk_u} = 0$, the first term in the supremum of \eqref{eq:mDphi} vanishes. Hence the conclusion follows.

(3): Generally, $\sup(F+G) \leq \sup F + \sup G$. However, if there exists an $f_0$ such that $F(f_0) = \sup F$ and $G(f_0) = \sup G$, then $\sup(F+G) \geq F(f_0) + G(f_0) = \sup F + \sup G$. Thus, the inequality becomes an equality, and the norm $\|B_{\bar{\vp}}\|^2$ is exactly the sum of the squared norms of the restricted operators.
\end{proof}

We now consider the case where $\vp \in H^\infty$ is a unimodular function. 

\begin{corollary}\label{cor:mDphi-analytic}
Let $\vp$ be an inner function. Then the minimum modulus of $D_\vp$ is given by
\begin{equation}
m(D_\vp) = m\left(\restr{T_{\bar{\vp}}}{\clk_u}\right).
\end{equation}
Furthermore, if $u$ is an inner factor of $\vp$, then $m(D_\vp) = 0$.
\end{corollary}

\begin{proof}
Since $\vp \in H^\infty$, we have $T_{\ol{u\vp}} = T^*_{u\vp } = T^*_\vp T^*_u$ by Proposition \ref{prop:toeplitz-identities}. For any $f \in \clk_u$, we have $T^*_u f = P( \bar{u}f ) = 0$. Consequently, $T_{\ol{u\vp}} f = 0$ for all $f \in \clk_u$. Substituting this into \eqref{eq:mDphi}, we get
\begin{align*}
m(D_\vp) &= \sqrt{ 1 - \sup_{f \in \clk_u, \|f\|=1} \left( 1 - \|T_{\bar{\vp}} f\|^2 \right) } \\
&= \sqrt{ 1 - \left( 1 - \inf_{f \in \clk_u, \|f\|=1} \|T_{\bar{\vp}} f\|^2 \right) } \\
&= \sqrt{ \inf_{f \in \clk_u, \|f\|=1} \|T_{\bar{\vp}} f\|^2 } = m\left(\restr{T_{\bar{\vp}}}{\clk_u}\right).
\end{align*}
If $u$ is an inner factor of $\vp$, by Proposition \ref{prop:kernel-toeplitz}, it follows that $\restr{T_{\bar{\vp}}}{\clk_u} = 0$, hence $m(D_\vp) = 0$.
\end{proof}

\begin{corollary}
Let $\vp \in L^\infty$ be a unimodular function. Then $m(D_\vp) = 1$ if and only if $\vp$ is a constant of modulus one.
\end{corollary}
\begin{proof}
    Follows easily.
\end{proof}

\subsection{Minimum modulus of \texorpdfstring{$B_\vp$}{Bv}}
This section is devoted to the study of the minimum modulus of the operator $B_\vp$, where $\vp\in L^\infty$ unimodular function, and in the case where $\vp$ is a bounded analytic function. We commence our analysis by recalling the following established result regarding the norm of a truncated Toeplitz operator with an analytic symbol (cf. \cite[Theorem 1.3 and Equation 2.9, p. 15]{Peller:Book}). For the sake of completeness, we provide a concise proof of this result.

\begin{lemma}\label{lem:Aphi-analytic}
Let $u$ be a nonconstant inner function, and let $\vp \in H^\infty$. The truncated Toeplitz operator $A_\vp: \clk_u \to \clk_u$, defined by $A_\vp f = P_{\clk_u}(\vp f)$. Then
\begin{equation}\label{eq:norm-Aphi}
\|A_\vp\| = \|H_{\bar{u}\vp}\| = \dist_{L^\infty}(\bar{u}\vp, H^\infty).
\end{equation}
In particular, $\|A_\vp\| = 0$ if and only if $\vp \in uH^\infty$.
\end{lemma}
\begin{proof}
Let $f \in \clk_u$. Since $H^2 = \clk_u \oplus uH^2$, we have $P_{\clk_u} = I_{H^2} - P_{uH^2}$. Therefore,
\[
A_\vp f = \vp f - P_{uH^2}(\vp f).
\]
 Since $\vp f \in H^2$, we have
\[
P_{uH^2}(\vp f) = M_u P(\bar{u} \vp f).
\]
This implies
\[
\begin{aligned}
A_\vp f &= u(\bar{u}\vp f) - u P(\bar{u}\vp f) \\
&= u \left( \bar{u}\vp f - P(\bar{u}\vp f) \right)\\
&= u \left((I - P)(\bar{u}\vp f) \right).
\end{aligned}
\]
Therefore,
\[
A_\vp f = M_u H_{\bar{u}\vp}f.
\]

We conclude that $\|A_\vp\| = \left\|\restr{H_{\bar{u}\vp}}{\clk_u}\right\|$. The Hankel operator $H_{\bar{u}\vp}$ vanishes on the subspace $uH^2$ because if $g = uh \in uH^2$, then $\bar{u}\vp g = \vp h \in H^2$, so $P_{-}(\vp h) = 0$. Since $H_{\bar{u}\vp}$ is zero on the orthogonal complement $\clk_u ^\perp$ of $\clk_u$ in $H^2$, we have
\[
\left\|\restr{H_{\bar{u}\vp}}{\clk_u}\right\| = \|H_{\bar{u}\vp}\|.
\]
Finally, Nehari's Theorem gives $\|H_{\bar{u}\vp}\| = \dist_{L^\infty}(\bar{u}\vp, H^\infty)$. The condition $\|A_\vp\|=0$ is equivalent to $\bar{u}\vp \in H^\infty$, or $\vp \in uH^\infty$.
\end{proof}

\begin{proposition}\label{prop:minmod:Bphi}
Let $\vp\in L^\infty$ such that $|\vp|=1$ a.e. on $\T$. The minimum modulus of the operator $B_\vp$ is given by
\begin{equation}\label{eq:minmod:Bphi}
m(B_\vp) = \sqrt{1 - \|A_\vp\|^2}.
\end{equation}
\end{proposition}

\begin{proof}
The square of the minimum modulus of $B_\vp$ is given by 
\[ m(B_\vp)^2 = \inf_{f \in \clk_u, \|f\|=1} \langle B_\vp^* B_\vp f, f \rangle. \]
From \eqref{eq:Bphi-Aphi}, we have
\[ m(B_\vp)^2 = \inf_{\|f\|=1} \langle (I_{\clk_u} - A_\vp^* A_\vp) f, f \rangle = \inf_{\|f\|=1} (1 - \|A_\vp f\|^2). \]
We obtain
\[ m(B_\vp)^2 = 1 - \|A_\vp\|^2. \]
Taking the square root yields the desired result.
\end{proof}

\begin{theorem}\label{thm:min-modulus-Bphi}
Let $u$ be a nonconstant inner function, and let $\vp \in H^\infty$. The minimum modulus of the operator $B_\vp: \clk_u \to \clk_u^\perp$ satisfies the following properties:
\begin{enumerate}
    \item $m(B_\vp) = m\left(\restr{T_{\bar{u}\vp}}{\clk_u}\right) \ge m(T_{\bar{u}\vp}).$
    
\item \begin{equation}
    m(B_\vp) \ge \sqrt{ (\essinf_{z \in \T} |\vp(z)|)^2 - \dist_{L^\infty}(\bar{u}\vp, H^\infty)^2 }.
    \end{equation}
    Furthermore, if $\vp$ is an inner function, then
    \begin{equation}
    m(B_\vp) = \sqrt{ 1 - \dist_{L^\infty}(\bar{u}\vp, H^\infty)^2 }.
    \end{equation}
\end{enumerate}
\end{theorem}

\begin{proof}
Let $\vp\in H^\infty$. Then

(1): By the operator identity $B_\vp = M_u \restr{T_{\bar{u}\vp}}{\clk_u}$ established in Lemma \ref{lem:norm-Bphi}, the minimum modulus is given by
\[ m(B_\vp) = \inf_{f \in \clk_u, \|f\|=1} \|M_u T_{\bar{u}\vp} f\|. \]
We obtain $m(B_\vp) = \inf_{f \in \clk_u, \|f\|=1} \|T_{\bar{u}\vp} f\|=m\left(\restr{T_{\bar{u}\vp}}{\clk_u}\right)$. The inequality $m(B_\vp) \geq m(T_{\bar{u}\vp})$ follows because the infimum over the subset $\clk_u$ is bounded below by the infimum over the full space $H^2$.

(2): Let $f\in \clk_u$. Since $M_\vp f =P_{\clk_u}(\vp f)+P_{\clk_u^\perp}(\vp f)= A_\vp f + B_\vp f$, we have $\|B_\vp f\|^2 = \|\vp f\|^2 - \|A_\vp f\|^2$. Taking the infimum over unit vectors $f \in \clk_u$, we obtain

\[ m(B_\vp)^2 = \inf_{f \in \clk_u, \|f\|=1} \left( \|\vp f\|^2 - \|A_\vp f\|^2 \right). \]
Using the property $\inf(X - Y) \ge \inf X - \sup Y$, we have
\[ m(B_\vp)^2 \ge \left( \inf_{\|f\|=1} \|\vp f\|^2 \right) - \|A_\vp\|^2. \]
Applying the pointwise bound $\inf \|\vp f\|^2 \ge (\essinf |\vp|)^2$ and the distance formula $\|A_\vp\| = \dist_{L^\infty}(\bar{u}\vp, H^\infty)$, we obtain the stated lower bound. If $\vp$ is inner, then $\|\vp f\| = 1$ for all unit vectors, and the inequality becomes an exact identity.
\end{proof}

\begin{corollary} \label{cor:shift-Bz}
Let $u$ be a nonconstant inner function. The minimum modulus of the operator $B_z: \clk_u \to \clk_u^\perp$ is given by
\[
m(B_z)=
\begin{cases}
0, & \dim \clk_u>1,\\ 
\sqrt{1-|u(0)|^2}, & \dim \clk_u=1.
\end{cases}
\]
\end{corollary}

\begin{proof}
Since $\vp(z) = z$ is an inner function, from Proposition \ref{prop:minmod:Bphi}, we have
\[ m(B_z) = \sqrt{1 - \|A_z\|^2}, \]
where $A_z = \restr{P_{\clk_u} M_z}{\clk_u}$ is the compressed shift operator on the model space. Using \cite[Theorem 3.3]{Bala2021}, the conclusion follows.
\end{proof}

\begin{remark}
If $\dim \clk_u=1$, using Theorem \ref{thm:minmod-compressed-shift} and Corollary \ref{cor:shift-Bz}, it follows that $m(B_z)^2+m(A_z)^2=1$. 
\end{remark}

For any $\lambda \in \D$, the Blaschke factor $b_\lambda$ is defined as
\[b_{\lambda}(z) = \frac{z - \lambda}{1 - \bar{\lambda}z}, \quad z \in \D.\]

\begin{proposition} \label{prop:dim-Bz-zero}
Let $u$ be a nonconstant inner function. Then $m(B_z) = 0$ if and only if $\dim \clk_u > 1$.
\end{proposition}

\begin{proof}
From \eqref{eq:minmod:Bphi}, we have $m(B_z)^2 = 1 - \|A_z\|^2$. The condition $m(B_z) = 0$ is equivalent to $\|A_z\| = 1$. By \cite[Theorem 3.3]{Bala2021}, we have $\|A_z\| = 1$ if and only if $\dim \clk_u>1$. In the case where $\dim \clk_u = 1$, the inner function is a single Blaschke factor $b_\lambda$, and $\|A_z\| = |\lambda| < 1$, which implies $m(B_z) > 0$.
\end{proof}

\end{document}